\documentclass{article}
\usepackage{fullpage}

\usepackage[utf8]{inputenc}
\usepackage{amsthm,amssymb}
\usepackage{amsmath}
\usepackage{thmtools,mathtools}
\usepackage{caption,subcaption}
\usepackage[dvipsnames]{xcolor}
\usepackage{epsfig,graphicx,graphics,color,tikz,tikz-cd,tcolorbox,stackengine}
\usepackage{enumerate,enumitem}
\usepackage[sort]{cite}
\usepackage{xspace}
\usepackage{bbm}
\usepackage{hyperref}
\usepackage[capitalise,noabbrev]{cleveref}
\hypersetup{
    colorlinks=true,       % false: boxed links; true: colored links
    linkcolor=blue,        % color of internal links (change box color with linkbordercolor)
    citecolor=magenta,         % color of links to bibliography
    filecolor=magenta,     % color of file links
    urlcolor=cyan,         % color of external links
    linktoc=all
}
\usepackage{titling}
\usetikzlibrary{calc,patterns}
\usepackage{anyfontsize}

\newtcolorbox{probbox}{arc=6pt,
                      colback=white!100,
                      colframe=black!50,
                      before skip=6pt,
                      after skip=6pt,
                      boxsep=1pt,
                      left=6pt,
                      right=6pt,
                      top=4pt,
                      bottom=4pt}

\newcommand{\decprob}[3]{
   \begin{center}%  
    \begin{minipage}{0.96\linewidth}%
      \begin{probbox}
      \textsc{#1}\\[0.2ex]
      \textbf{Input:} #2\\[0.2ex]
      \textbf{Question:} #3
      \end{probbox}
    \end{minipage}%
  \end{center}
}

\newcommand{\searchprob}[3]{
   \begin{center}%  
    \begin{minipage}{0.96\linewidth}%
      \begin{probbox}
      \textsc{#1}\\[0.2ex]
      \textbf{Input:} #2\\[0.2ex]
      \textbf{Goal:} #3
      \end{probbox}
    \end{minipage}%
  \end{center}
}

\definecolor{greenish}{RGB}{27,158,119}
\definecolor{MyOrange}{RGB}{217,95,2}
\definecolor{MyPurple}{RGB}{117,112,179}

\newcommand{\suppp}{{\rm supp}\sp{+}}
\newcommand{\suppm}{{\rm supp}\sp{-}}
\newcommand{\MN}{M^\natural}
\def\NP{\mathsf{NP}}
\newcommand{\cH}{\mathcal{H}}

\theoremstyle{plain}
\newtheorem{thm}{Theorem}[section]
\newtheorem{lem}[thm]{Lemma}
\newtheorem{cor}[thm]{Corollary}

\newtheorem{prop}[thm]{Proposition}

\theoremstyle{definition}

\title{On the Complexity of Inverse Bivariate Multi-unit Assignment Valuation Problems}

\thanksmarkseries{alph}
\author{
Kristóf Bérczi\thanks{MTA-ELTE Matroid Optimization Research Group and HUN-REN–ELTE Egerváry Research Group, Department of Operations Research, Eötvös Loránd University, Budapest, Hungary. Email: \texttt{kristof.berczi@ttk.elte.hu}.}
\and
Lydia Mirabel Mendoza-Cadena\thanks{MTA-ELTE Matroid Optimization Research Group, Department of Operations Research, Eötvös Loránd University, Budapest, Hungary. Email: \texttt{lyd21@student.elte.hu}.}}

\date{}

\begin{document}
\maketitle
\tableofcontents

%%%%%%%%%%%%%%%%%%%%%%%%%%%%%%%%
 \newpage
% \pagenumbering{roman}
% \tableofcontents
% \newpage
% \pagenumbering{arabic}
% \setcounter{page}{1}
%%%%%%%%%%%%%%%%%%%%%%%%%%%%%%%%

%%%%%%%%%%%%%%%%%%%%%%%%%%%%%%%%
\begin{abstract} 
 Inverse and bilevel optimization problems play a central role in both theory and applications. These two classes are known to be closely related due to the pioneering work of Dempe and Lohse (2006), and thus have often been discussed together ever since. In this paper, we consider inverse problems for multi-unit assignment valuations. Multi-unit assignment valuations form a subclass of strong-substitutes valuations that can be represented by edge-weighted complete bipartite graphs. These valuations play a key role in auction theory as the strong substitutes condition implies the existence of a Walrasian equilibrium. A recent line of research concentrated on the problem of deciding whether a bivariate valuation function is an assignment valuation or not. In this paper, we consider an \emph{inverse} variant of the problem: we are given a bivariate function $g$, and our goal is to find a bivariate multi-unit assignment valuation function $f$ that is as close to $g$ as possible. The difference between $f$ and $g$ can be measured either in $\ell_1$- or $\ell_\infty$-norm. Using tools from discrete convex analysis, we show that the problem is strongly $\NP$-hard. On the other hand, we derive linear programming formulations that solve relaxed versions of the problem. 

\medskip

\noindent \textbf{Keywords:} Assignment valuations, Inverse problems, Linear programming, $M^\natural$-concave sets, Product-mix auction

\end{abstract}
%%%%%%%%%%%%%%%%%%%%%%%%%%%%%%%%

%-----------------------------------------------------------------------
\section{Introduction}
%-----------------------------------------------------------------------

Multi-unit assignment valuations generalize single-unit ones by extending their domain from zero-one vectors to non-negative integer vectors. As a result, the gross-substitutes condition \cite{kelso1982job,gul1999walrasian} for single-unit assignment valuations transforms into the so-called \emph{strong-substitutes} condition \cite{murota2003discrete}, which in turn implies strong structural properties, e.g. it ensures the existence of a Walrasian equilibrium in the market. For sake of simplicity, we refer to the multi-unit version of these functions as \emph{assignment valuations} throughout since this may cause no confusion.

Motivated by the `product-mix' auction used in Bank of England \cite{klemperer2010auction}, Otsuka and Shioura~\cite{otsuka2023assignment} considered the so-called Bivariate Assignment Valuation Checking Problem (BAVCP): Given a bivariate function $g$, decide if $g$ is an assignment valuation. They developed an algorithm for the problem that, if the answer is `\texttt{Yes}', outputs a weighted bipartite graph representing $g$. Their approach is based on characterizing assignment valuations in terms of maximizer sets and uses basic tools of discrete convex analysis. 
However, for an instance of BAVCP, the function $g$ not being an assignment valuation may be caused by data inaccuracies in the input, e.g., when the values are coming from certain physical measurements. The domain of the function increases quadratically in the total supply, while a `\texttt{No}' answer may stem even from a very small subset of wrong data. Therefore, the fact that $g$ is not an assignment valuation allows for little to no conclusions. This motivates the study of the problem in the context of inverse optimization.

\emph{Inverse problems} form a rich class of questions that has significant interest in the past decades due to its applicability in both theory and practice, see~\cite{richter2016inverse} for an introduction. These problems can be categorized into two main classes. In \emph{parameter estimation}, the goal is to determine certain parameters of a system that are consistent with a set of observations -- an illustrative example is seismic tomography, see e.g.~\cite{nolet1987seismic}. In \emph{solution imposition}, the goal is to modify the system parameters as little as possible so as to enforce a set of solutions -- a classical example is the inverse shortest path problem studied by Burton and Toint~\cite{burton1992instance}. 

A relation between inverse problems and \emph{bilevel programming} was first established in the 2000s, when Dempe and Lohse~\cite{dempe2006Inverse} deduced the optimality conditions for an inverse linear programming problem by modelling it as a bilevel programming question. Bilevel optimization programs are used to model decision processes with two decision makers: the \emph{leader} modifies the problem of the \emph{follower} with her decision, but the value of this decision is influenced by the decision of the follower. The task is to determine a strategy for the leader that optimizes the objective value. However, the follower's strategy for breaking ties among optimal solutions is not specified, though it heavily affects the leader's objective. In a \emph{pessimistic} formulation, the follower behaves non-cooperatively and takes a solution against the leader. In an \emph{optimistic} formulation, the follower behaves in a cooperative way and takes a solution that is most beneficial for the leader. Dempe and Zemkoho~\cite[Chapter 20]{dempe2020bilevel} explained that for optimistic formulations, the upper-level objective function is often of target type and chosen in order to reconstruct parameters in the lower-level parametric optimization problem from given observations. This allows the hierarchical model to be considered as an inverse optimization problem, which  often happens in hyperparameter learning and bilevel optimal control~\cite[Chapters 6, 16]{dempe2020bilevel}. Moreover, Sabach and Shtern~\cite{Sabach2017first} considered convex bilevel optimization problems and used their algorithm for solving a linear inverse problem as an example.

Let us briefly review some basic complexity results regarding inverse problems and bilevel optimization. It is known that an inverse problem can be $\NP$-hard, even if its underlying optimization problem is in $\mathsf{P}$, as shown by Cai,  Yang, and Zhang~\cite{cai1999complexity}. In bilevel optimization, Jeroslow~\cite{jeroslow1985polynomial} showed that the linear bilevel program, which is linear in both the upper- and lower-levels and has continuous variables, is $\NP$-hard, and a simpler proof was later given by Ben-Ayed and Blair~\cite{ben1990computational}. This result was further strengthened by Hansen, Jaumard and Savard~\cite{hansen1992new} who proved strong $\NP$-hardness. In fact, somewhat unexpectedly compared to other areas of optimization, bilevel optimization problems may be infeasible even when all functions are continuous over compact sets~\cite[Chapter 10]{dempe2020bilevel}. Furthermore,  Vicente, Savard and Júdice~\cite{vicente1994descent} proved that verifying local optimality in linear bilevel programming is an $\NP$-hard problem. Therefore, it is quite common to set assumptions that guarantee the existence of solutions. For example, the problem is in $\mathsf{P}$ when the number of variables for the follower is bounded by a constant, as shown by Deng~\cite{deng1998complexity}.

\medskip

The scope of this work is to study an inverse counterpart BAVCP as well as some relaxations of the problem. Assignment valuations can be represented by edge-weighted complete bipartite graphs; for the precise definition, see Section~\ref{sec:preliminaries}. In real world problems, the vertex classes (i.e. the sets of `agents' and `goods') and the supplies for the agents are usually known while the weight function is only estimated, which may lead to inaccuracies. This motivates the following problem, which is a natural inverse counterpart of BAVCP: Given a bivariate function $g$ and a set of agents together with their supplies, find a weight function for the agent-item pairs such that the induced assignment valuation $f$ is as \emph{close} to the input function $g$ as possible. The difference between $f$ and $g$ can be measured in several ways; in this paper, we focus on the $\ell_1$- and $\ell_\infty$-norm objectives.

\medskip

The paper is structured as follows. Section~\ref{sec:preliminaries} summarizes basic definitions and notation, together with an overview of structural results on $\MN$-sets and $\MN$-concave functions. In Section~\ref{sec:irred_repr}, we analyze properties of complete weighted bipartite graphs that represent the same assignment valuation. We prove strongly $\NP$-hardness of the proposed inverse problem in Section~\ref{sec:inverse_with_supply}. Finally, Section~\ref{sec:inverse_hexagons} presents a relaxation.

%-----------------------------------------------------------------------
\section{Preliminaries}
\label{sec:preliminaries}
%-----------------------------------------------------------------------

The purpose of this section is to give an overview of basic definitions and notation as well as introducing basic properties of $\MN$-convex sets and $\MN$-concave functions; for further details, we refer the interested reader to \cite{murota2003discrete,otsuka2023assignment}.

%-----------------------------------------------------------------------
\subsection{Basic notation} 
\label{sec:basicNotation}
%-----------------------------------------------------------------------

The sets of \emph{real}, \emph{integer}, \emph{non-negative integer}, and \emph{positive integer numbers} are denoted by $\mathbb{R}$, $\mathbb{Z}$, $\mathbb{Z}_{\geq 0}$ and $\mathbb{Z}_{> 0}$, respectively. For a positive integer $n$, we use $[n]\coloneqq \{1,\dots,n\}$. Given a vector $x\in\mathbb{R}^n$, we denote its \emph{$i$th coordinate} by $x^i$. We further use $\suppp(x)=\{i\in[n] \mid x^i>0 \}$ and $\suppm(x)=\{i\in[n] \mid x^i<0 \}$ for denoting the \emph{sets of positive} and \emph{negative} coordinates of $x$, respectively. For an index $i\in[n]$, the \emph{characteristic vector of $i$} is denoted by $\chi_i\in\{0,1\}^n$, that is, $\chi_i^j=1$ if $i=j$ and $0$ otherwise. For ease of notation, we use $\chi_0$ to denote the \emph{all-zero vector}. For a set $S\subseteq\mathbb{Z}^n$ and function $f\colon S \to \mathbb{R} \cup \{-\infty\}$, we use ${\rm dom} f=\{x \in S | f(x)>-\infty \}$. For a non-negative integer $\Phi\in\mathbb{Z}_{> 0}$, we use $T_\Phi = \{ x \in \mathbb{Z}^2_{\geq 0} \, \vert \,  x^1 + x^2 \leq \Phi \}$. Throughout the paper, we use $N=\{1,2\}$. Given a family $\cH$ of sets, we denote by $\cH^\cup$ the \emph{union closure} of $\cH$ defined as $\cH^\cup = \{ \bigcup_{H\in\cH'} H \mid \cH'\subseteq \cH \}$.

%-----------------------------------------------------------------------
\subsection{3-partition problem} 
%-----------------------------------------------------------------------

In the 3-partition problem (\textsc{$3$-Partition}), the goal is to decide whether a given set of integers can be partitioned into triplets all having the same sum. 

\decprob{$3$-Partition}{A set $A=\{a_1,\dots,a_{3k}\}$ of $3k$ different positive integers, and $T=(\sum_{i=1}^{3k}a_i)/k$.}{Decide if $A$ can be partitioned into $k$ disjoint triplets $A_1,\dots,A_k$ such that $\sum_{a\in A_i}a=T$ for $i\in[k]$.}

\textsc{$3$-Partition} is known to be strongly $\NP$-complete~\cite[Problem SP15]{garey1983computers}, that is, the problem is $\NP$-complete even if the the numbers in the input are represented in unary. It is also known that the problem remains strongly $\NP$-complete if $T/4 < a_i < T/2$ holds for $i\in[3k]$. 

%-----------------------------------------------------------------------
\subsection{\texorpdfstring{$\MN$}{Mn}-convex sets} 
\label{sec:convex-sets}
%-----------------------------------------------------------------------

To make the paper self-contained, we repeat the most fundamental definitions and results from \cite{otsuka2023assignment}. A set $S\subseteq \mathbb{Z}^n$ is \emph{$\MN$-convex} if for every $x,y\in S$ and $i\in\suppp(x-y)$ there exists $j\in\suppm(x-y)\cup\{0\}$ such that both $x-\chi_i+\chi_j$ and $y+\chi_i-\chi_j$ are in $S$. This property is usually called the \emph{exchange property}. In two dimensions, $\MN$-convex sets have a simple characterization, as shown by the following folklore result.

\begin{prop}[{\cite[Proposition 2.3]{otsuka2023assignment}}] \label{prop:sets_as_ineq}
A bounded set $S \subseteq \mathbb{Z}^2$ is $\MN$-convex if and only if it can be represented by a system of inequalities of the following form:
\begin{align*}
        S = \{ x \in \mathbb{Z}^2 \, \vert \, \lambda_1 \leq x^1 \leq \mu_1,\ \lambda_2 \leq x^2 \leq \mu_2,\ \lambda_0 \leq x^1 + x^2 \leq \mu_0 \}. \label{eq:convex_sets_inequalities}
\end{align*}
\end{prop}
We may assume that the inequalities are tight, that is, $\lambda_i=\min\{x^i\mid x\in S\}$, $\mu_i=\max\{x^i\mid x\in S\}$ for $i=1,2$, and $\lambda_0=\min\{x^1+x^2\mid x\in S\}$, $\mu_0=\max\{x^1+x^2\mid x\in S\}$. By Proposition~\ref{prop:sets_as_ineq}, for the \emph{convex hull} $\operatorname{conv}(S)$ of $S$, we have $\operatorname{conv}(S) \cap \mathbb{Z}^2 =S$. Furthermore, $\operatorname{conv}(S)$ is the convex hull of the vertices 
\begin{equation*}
    (\lambda_1, \mu_2), \, (\mu_0 - \mu_2, \mu_2), \, (\mu_1, \mu_0 - \mu_1), \, (\mu_1,\lambda_2), \, (\lambda_0 - \lambda_2, \lambda_2), \, (\lambda_1, \lambda_0 - \lambda_1),
\end{equation*}
and so $\operatorname{conv}(S)$ is a hexagon; see Figure~\ref{fig:MN_convex_sets} for an example. Note that some vertices of the hexagon may coincide, in which case we call the hexagon \emph{degenerate}. The \emph{edges} of $S$ and their \emph{lengths} are then defined as
\begin{itemize}\itemsep0em
    \item upper-horizontal (UH) edge, with $\ell_{\text{UH}}(S) =(\mu_0 - \mu_2) - \lambda_1 $,
    \item  lower-horizontal (LH) edge, with $\ell_{\text{LH}}(S) = \mu_1 - (\lambda_0 - \lambda_2)$,
    \item left-vertical (LV) edge, with $\ell_{\text{LV}}(S) = \mu_2 - (\lambda_0 - \lambda_1)$, 
    \item right-vertical (RV) edge, with $\ell_{\text{RV}}(S) = (\mu_0 - \mu_1) - \lambda_2$,
    \item upper-right-diagonal (URD) edge, with $\ell_{\text{URD}}(S) = \sqrt{2} ( \mu_1 - (\mu_0 - \mu_2) ) $,
    \item lower-left-diagonal (LLD) edge  with $\ell_{\text{LLD}}(S) = \sqrt{2} (( \lambda_0 - \lambda_2) - \lambda_1) $.       
\end{itemize}

\begin{figure}[t]
    \centering
    \begin{tikzpicture}[scale=.7]
        \node[fill,circle,scale=0.7, label={above:\scriptsize$(\lambda_1, \mu_2)$}] (h1) at (1,5) { };
        \node[fill,circle,scale=0.7,label={above right:\scriptsize$(\mu_0 - \mu_2, \mu_2)$}] (h2) at (4,5) {};
        \node[fill,circle,scale=0.7,label={above right:\scriptsize$(\mu_1, \mu_0 - \mu_1)$}] (h3) at (6,3) {};
        \node[fill,circle,scale=0.7,label={below right:\scriptsize$(\mu_1,\lambda_2)$}] (h4) at (6,1) {};
        \node[fill,circle,scale=0.7,label={below left:\scriptsize$(\lambda_0 - \lambda_2, \lambda_2)$}] (h5) at (2,1) {};
        \node[fill,circle,scale=0.7,label={above left:\scriptsize$(\lambda_1, \lambda_0 - \lambda1)$}] (h6) at (1,2) {};
        
        \draw[MyPurple, line width = 1.5pt] (h1)--(h2) node[midway,above] {\textbf{UH}};
        \draw[MyPurple, line width = 1.5pt] (h4)--(h5) node[midway,below] {\textbf{LH}};
        \draw[greenish, line width = 1.5pt] (h1)--(h6) node[midway,left] {\textbf{LV}};
        \draw[greenish, line width = 1.4pt] (h3)--(h4) node[midway,above right] {\textbf{RV}};
        \draw[MyOrange, line width = 1.5pt] (h6)--(h5) node[midway,left] {\textbf{LLD}};
        \draw[MyOrange, line width = 1.5pt] (h2)--(h3) node[midway,above right] {\textbf{URD}};
            
        \foreach \x in {0,1,2,3,4,5,6,7}{%
            \foreach \y in {0,1,2,3,4,5,6}
                {\draw [fill = black] (\x,\y) circle (2pt);}
        };
    \end{tikzpicture}
    \caption{An $\MN$-convex set of positive-type with edge and vertex labels.}
    \label{fig:MN_convex_sets}
\end{figure}
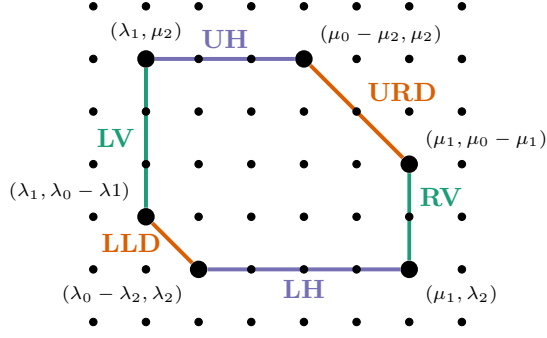

Otsuka and Shioura \cite{otsuka2023assignment} observed that for a two-dimensional $\MN$-convex set $S \subseteq \mathbb{Z}^2$, we have $\ell_{\text{LH}}(S) - \ell_{\text{UH}}(S) =\ell_{\text{LV}}(S) - \ell_{\text{RV}}(S) =\frac{1}{\sqrt{2}} ( \ell_{\text{URD}}(S) - \ell_{\text{LLD}}(S))$. This result leads to a key definition: we say that $S$ is of \emph{positive-type} if  $\ell_{\text{LH}}(S) - \ell_{\text{UH}}(S) > 0$, $S$ is of  \emph{negative-type} if $\ell_{\text{LH}}(S) - \ell_{\text{UH}}(S) < 0 $, and $S$ is of \emph{zero-type} if $\ell_{\text{LH}}(S) - \ell_{\text{UH}}(S) = 0$.

%-----------------------------------------------------------------------
\subsection{\texorpdfstring{$\MN$}{Mn}-concave functions}
\label{sec:concave-functions}
%-----------------------------------------------------------------------

M${}^\natural$-concave functions are introduced by Murota and Shioura~\cite{MurotaShioura1999} and play a central role in the theory of discrete convex analysis. For a set $S\subseteq\mathbb{Z}^n$, a function $f\colon S \to \mathbb{R} \cup \{-\infty\}$ is said to be {\em M${}^\natural$-concave} if it satisfies the following \emph{exchange property}: 
\begin{description}
\item[(M${}^\natural$-EXC)] \label{exc}
$\forall x,y \in {\rm dom} f,\ \forall i \in \suppp(x-y),\ \exists j \in \suppm(x-y) \cup \{0\}$: 
$$
f(x)+f(y) \le f(x-\chi_i+\chi_j) + f(y+\chi_i-\chi_j).  
$$
\end{description}

For a positive integer $\Phi\in\mathbb{Z}_{>0}$, let $T_\Phi = \{ x \in \mathbb{Z}^2_{\geq 0} \, \vert \,  x^1 + x^2 \leq \Phi \}$. A function $f\colon T_\Phi\to\mathbb{R}$ can be regarded as a function on $\mathbb{Z}^2$ by setting $f(x)=-\infty$ for $x \in \mathbb{Z}^2 \setminus T_\Phi$. For simplicity, we say that $f$ is $\MN$-concave if this extended function is $\MN$-concave. For bivariate $\MN$-concave functions, the next characterization holds.

\begin{prop}[{\cite[Proposition 2.2]{otsuka2023assignment}}]\label{prop:MN_concave_function_ineq_characterization}
A bivariate function $f\colon  T_\Phi \to \mathbb{R}$ is $\MN$-concave if and only if the following hold for every $k,h\in\mathbb{Z}_{\geq 0}$ with $k+h\leq \Phi-2$:
    \begin{align*}
        f(k,h) + f(k+1,h+1) &\leq f(k+1,h) + f(k,h+1), \\%\label{eq:MN_Concave_cond1} \\
        f(k,h+1) + f(k+2,h) &\leq f(k+1,h+1) + f(k+1,h), \\%\label{eq:MN_Concave_cond2}\\
        f(k+1,h) + f(k,h+2) &\leq f(k+1,h+1) + f(k,h+1). %\label{eq:MN_Concave_cond3}
    \end{align*}
\end{prop}

For a bivariate function $f\colon  T_\Phi \to \mathbb{R}$ and a vector $p \in \mathbb{R}^2$, we define the \emph{maximizer set} $D_f(p) \subseteq T_\Phi$ as
\begin{equation}
    D_f(p) \coloneqq \{ x \in T_\Phi \, \vert \, f(x) - \langle p,x \rangle  \geq f(y) - \langle p,y \rangle\ \text{for every $y \in T_\Phi$} \}. \label{eq:maximizerD}
\end{equation}
If $D \subseteq T_\Phi$ is a two-dimensional maximizer set of $f$, then there exists a unique vector $p \in \mathbb{R}^2$ so that $D=D_f(p)$. The vector $p$ is called the \emph{slope vector} of $D$. Then, $\MN$-concave functions can be characterized also in terms of maximizer sets~\cite[Proposition 2.5]{otsuka2023assignment}: a bivariate function $f\colon  T_\Phi \to \mathbb{R}$ is $\MN$-concave if and only if $D_f(p)$ is an $\MN$-convex set for every $p \in \mathbb{R}^2$. For an $\MN$-concave function $f$, we denote its \emph{family of two-dimensional maximizers} by $\mathcal{D}_f$.

%-----------------------------------------------------------------------
\subsection{Multi-unit assignment valuations}
\label{sec:assignment_valuations}
%-----------------------------------------------------------------------

We consider combinatorial markets consisting of a set $V=[n]$ of agents, a set $N=\{1,2\}$ of two multiple indivisible goods, weight function $w\colon V\times N\to\mathbb{R}$, and a strictly positive supply function $\varphi\colon V\to\mathbb{Z}_{>0}$ that describes the agents' maximum demands. Note that such a market can be represented by an edge-weighted complete bipartite graph $G=(V,N; V \times N)$ with weight function $w:V \times N \to \mathbb{R}$ and supply function $\varphi:V\to \mathbb{Z}_{>0}$. Let $\Phi = \sum_{i \in V} \varphi (i)$. Then, the corresponding \emph{assignment valuation} is $f\colon T_\Phi\to\mathbb{R}$, where
\begin{align}
    f(x) = \max  \bigg\{ \sum_{i \in V} w(i,1)\cdot y^1_i + w(i,2)\cdot y^2_i\,\bigg\vert\, & \sum_{i \in V} y^j_i = x^j\quad \text{for $j\in N$}, \nonumber \\
    & y^1_i + y^2_i \leq \varphi(i)\quad \text{for $i \in V$}, \nonumber  \\
    & y^j_i \in \mathbb{Z}_{\geq 0}\quad \text{for $j\in N$, $i\in V$}
    \bigg\}. \label{eq:assignment_valuation_def}
\end{align}
In other words, for $x\in T_\Phi\subseteq \mathbb{Z}_{> 0}^2$, the value $f(x)$ equal the maximum $w$-weight of a subset $F\subseteq V\times N$ of edges such that every vertex $i\in V$ has degree at most $\varphi(i)$ and every vertex $j\in N$ has degree exactly $x^j$ in $F$. For a fixed $x\in T_\Phi$, computing the value of $f(x)$ is called a \emph{maximum weight $b$-matching problem} that can be solved efficiently; see~\cite[Chapter 31]{schrijver2003combinatorial} for further details. We say that $(G,w,\varphi)$ is a \emph{representation} of the function $f$. Assignment valuations are known to be $\MN$-concave, see e.g.~\cite{murota2022discrete, murota2022discreteConvex}. Otsuka and Shioura~\cite{otsuka2023assignment} gave necessary and sufficient conditions for a function to be an assignment valuation. 

\begin{prop}[{\cite[Corollary 3.5]{otsuka2023assignment}}] \label{prop:characterization_assig_val}
    An $\MN$-concave function $f\colon T_\Phi \to \mathbb{R}$ with $f(0,0) = 0$ is an assignment valuation if and only if every two-dimensional maximizer set of $f$ is an $\MN$-convex set of positive- or zero-type, and there exists at least one having positive-type.
\end{prop}

%-----------------------------------------------------------------------
\section{Inverse problems}
\label{sec:inverse_with_supply}
%-----------------------------------------------------------------------

Suppose that the values of a function $g\colon T_\Phi \to \mathbb{R}$ are coming from certain measurements, but the function turns out not to be an assignment valuation due to inaccuracies in the data. In such a scenario, one might be interested in finding an assignment valuation $f: T_\Phi\to\mathbb{R}$ that is as close to $g$ as possible. The difference between $f$ and $g$ can be measured in various ways, here we concentrate on the $\ell_1$- and $\ell_\infty$-norms. That is, our goal is to minimize $\|f-g\|_1=\sum_{x\in T_\Phi}|f(x)-g(x)|$ or $\|f-g\|_\infty=\max_{x\in T_\Phi}|f(x)-g(x)|$. For sake of simplicity, when it may cause no confusion, we use $\|\cdot\|$ to denote any of these norms.

Since we consider algorithmic problems related to bivariate assignment valuation functions, we have to specify how such a function is given as an input. Throughout the section, we assume that any function $g\colon T_\Phi\to\mathbb{R}$ appearing as the input of a problem is given by explicitly listing its value, that is, by the set $\{(x,g(x))\mid x\in T_\Phi\}$.

In Section~\ref{sec:irred_repr}, we introduce the notion of irreducible representations and overview their properties. In Section~\ref{sec:inverse_with_supply_and_func}, relying on the special structure of such representations, we show that finding an assignment valuation $f$ that minimizes $\|f-g\|$ is strongly $\NP$-hard even if only the weight function is missing. Finally, we explain how to solve the problem if $f$ is only required to be an $\MN$-concave function in Section~\ref{sec:mn}.

%-----------------------------------------------------------------------
\subsection{Irreducible representations}
\label{sec:irred_repr}
%-----------------------------------------------------------------------

Let $\Phi\in\mathbb{Z}_{>0}$ and let $f\colon T_\Phi \to \mathbb{R}$  be a bivariate assignment valuation represented by $(G,w,\varphi)$. If $w(u,1) = w(v,1)$ and $w(u,2) = w(v,2)$ holds for a pair of vertices $u,v \in V$, then let $G'=(V',N; V' \times N)$ denote the complete bipartite graph obtained from $G$ by deleting the vertex $v$, that is, let $V'\coloneqq V\setminus\{v\}$, and define $w'\colon V'\times N\to\mathbb{R}$ and $\varphi'\colon V'\to\mathbb{R}$ by setting $w'(i,j)=w(i,j)$  for $i \in V'$ and $j \in N$, and $\varphi'(i)=\varphi(i)$ if $i\neq u$ and $\varphi'(u)=\varphi(u)+\varphi(v)$ otherwise. For sake of simplicity, we call this operation the \emph{contraction} of vertices $u$ and $v$. It is not difficult to see that the assignment valuation represented by $(G',w',\varphi')$ is identical to the original one. If no contraction step can be applied, then 
\begin{equation}
    \forall i, j \in V,i\neq j:\ \{w(i,1),w(i,2)\}\neq \{w(j,1),w(j,2)\} . \tag{IR}\label{eq:distinct_weights}
\end{equation}
We call a representation $(G,w,\varphi)$ of $f$ \emph{irreducible} if it satisfies this property. Clearly, starting from any representation $(G,w,\varphi)$, one can obtain an irreducible representation $(G',w',\varphi')$ by the repeated application of the contraction operation. In particular, Otsuka and Shioura~\cite{otsuka2023assignment} showed the following.

\begin{prop}[{\cite[Lemma 3.2]{otsuka2023assignment}}]\label{prop:relation_w-p_and_varphi-edgeLength}
    Let $(G,w,\varphi)$ be an irreducible representation of a bivariate assignment valuation $f$. Let $p\in \mathbb{R}^2$ be a vector such that the maximizer set $D_f(p)$ is two-dimensional.
    \begin{enumerate}
        \item If $p\neq ( w(i,1), w(i,2) )$ for all $i \in V$, then $D_f(p)$ is of zero-type.
        \item If $p= ( w(i_p,1), w(i_p,2) )$ for some  $i_p \in V$, then $D_f(p)$ is of positive-type. Furthermore, $i_p$ is the unique vertex with $p= ( w(i_p,1), w(i_p,2) )$, and it satisfies $\ell_{\text{LH}}(D_f(p)) - \ell_{\text{UH}}(D_f(p)) =  \varphi(i_p)$.
    \end{enumerate}
\end{prop}

Proposition~\ref{prop:relation_w-p_and_varphi-edgeLength} has a simple but interesting corollary.

\begin{cor} \label{cor:unique_irred_sol}
    Let $\Phi\in\mathbb{Z}_{> 0}$ and let $f\colon T_\Phi \to \mathbb{R}$ be a bivariate assignment valuation. Then, $f$ has a unique irreducible representation.
\end{cor}
\begin{proof}
    By Proposition~\ref{prop:relation_w-p_and_varphi-edgeLength}, in any irreducible representation $(G,w,\varphi)$ of $f$, the weights and the supplies are uniquely determined (modulo reordering the vertices of $G$) by the family of two-dimensional maximizer sets of $f$ having positive type.
\end{proof}

In~\cite{otsuka2023assignment}, Otsuka and Shioura developed a simple algorithm that determines the family of two-dimensional maximizers of an $\MN$-concave function together with their corresponding slope vectors. This in turn leads to an algorithm that returns an irreducible representation of a bivariate assignment valuation function in $O(\Phi)$ time.

%-----------------------------------------------------------------------
\subsection{Hardness results}
\label{sec:inverse_with_supply_and_func}
%-----------------------------------------------------------------------

First, we consider a natural decision problem that aims at determining the weights in a representation of a bivariate assignment valuation function. Somewhat surprisingly, the problem turns out to be $\NP$-complete.

\decprob{Determine-Weight}{A bivariate assignment valuation $f\colon T_\Phi \to \mathbb{R}$, a graph $G=(V,N;V\times N)$, and a supply function $\varphi\colon V \to \mathbb{R}_{> 0}$ with $\sum_{i\in V}\varphi(i)=\Phi$.}{Decide if there exists a weight function $w\colon V\times N \to \mathbb{R}$ such that $(G,w,\varphi)$ is a representation of $f$.}

It is worth emphasizing the subtle difference between \textsc{Determine-Weight} and the problem solved by the algorithm of Otsuka and Shioura. In the latter, the graph $G$ and the supply function $\varphi$ are not fixed, and the algorithm determines an irreducible representation of the function. As we will see, the hardness of \textsc{Determine-Weight} stems from the fact that an assignment valuation usually have numerous -- non-irreducible -- representations.

\begin{thm}\label{thm:get_w_is_hard}
    \textsc{Determine-Weight} is strongly $\NP$-complete.
\end{thm}
\begin{proof}
Given a weight function $w\colon V\times N\to\mathbb{R}$, one can compute the value of the function represented by $(G,w,\varphi)$ and compare it to $f(x)$ for every $x\in T_\Phi$, hence the problem is clearly in $\NP$.

We prove hardness by reduction from \textsc{3-partition}. Let $A=\{a_1,\dots,a_{3k}\}$ be an instance of \textsc{$3$-partition} such that the numbers are represented in unary and $T/4 < a_i < T/2$ for $i\in[3k]$, where $T=(\sum_{i=1}^{3k}a_i)/k$. We create an instance $(f,G,\varphi)$ of \textsc{Determine-Weight} as follows. With the choice $\Phi=\sum_{i=1}^{3k} a_i$, let $f\colon T_\Phi\to\mathbb{R}$ denote the bivariate assignment valuation represented by $(G',w',\varphi')$, where $G'=(V',N;V'\times N)$ with $V'=[k]$, $w'(i,j)=i$ for $(i,j)\in V'\times N$, and $\varphi'(i)=T$ for $i\in[k]$. Note that $|T_\Phi| \leq (\sum_{i=1}^{3k} a_i)^2$, hence the value of $f(x)$ can be computed for all $x\in T_\Phi$ in polynomial time as the input of \textsc{3-partition} is represented in unary. Finally, let $G=(V,N;V\times N)$ with $V=[3k]$ and set $\varphi(i) = a_i$ for $i\in[3k]$; see Figure~\ref{fig:reduction} for an example. 

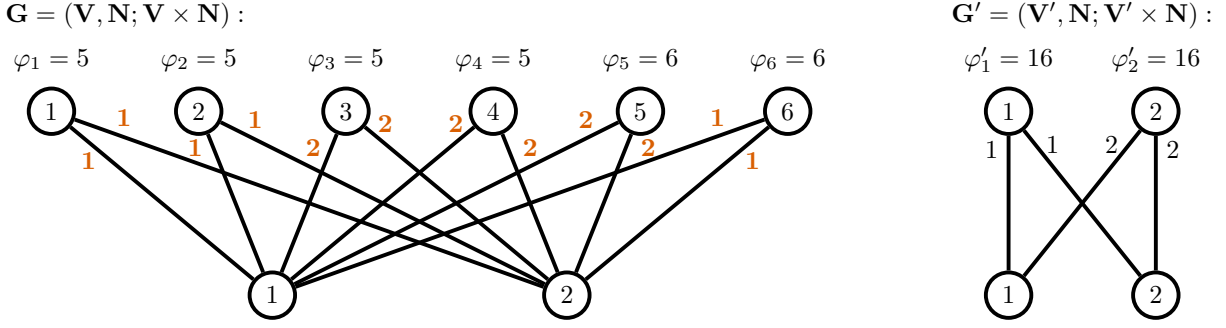
\begin{figure}
    \centering
    \resizebox{\textwidth}{!}{
    \begin{tikzpicture}[
    nodeSty/.style={circle, draw, line width = 1.5pt},
    edgeSty/.style={line width = 1.5pt}]
        \node (G) at (3,3.3) {$\mathbf{G=(V,N; V\times N)}:$};
        \node[nodeSty] (v1) at (2,2) {1};
        \node[nodeSty] (v2) at (4,2) {2};
        \node[nodeSty] (v3) at (6,2) {3};
        \node[nodeSty] (v4) at (8,2) {4};
        \node[nodeSty] (v5) at (10,2) {5};
        \node[nodeSty] (v6) at (12,2) {6};
        \node[nodeSty] (n1) at (5,-0.5) {1};
        \node[nodeSty] (n2) at (9,-0.5) {2};    
        \node (varphi1) at (2,2.7) {$\varphi_1 = 5$};
        \node (varphi2) at (4,2.7) {$\varphi_2 = 5$};
        \node (varphi3) at (6,2.7) {$\varphi_3 = 5$};
        \node (varphi4) at (8,2.7) {$\varphi_4 = 5$};
        \node (varphi5) at (10,2.7) {$\varphi_5 = 6$};
        \node (varphi6) at (12,2.7) {$\varphi_6 = 6$};
    
        \foreach \x \w \sideOne \sideTwo in {1/1/below/above, 2/1/left/above, 3/2/left/above,4/2/above/right, 5/2/above/right, 6/1/above/below}{
            \draw (v\x) edge[edgeSty]  node[\sideOne, pos=0.1, MyOrange] {$\mathbf{\w }$} (n1);
            \draw (v\x) edge[edgeSty]  node[\sideTwo, pos=0.1, MyOrange] {$\mathbf{\, \w}$} (n2);
        }
        
        \node (Gp) at (3+13,3.3) {$\mathbf{G'=(V',N; V'\times N):}$};
        \node[nodeSty] (vp1) at (2+13,2) {1};
        \node[nodeSty] (vp2) at (4+13,2) {2};
        \node[nodeSty] (np1) at (2+13,-0.5) {1};
        \node[nodeSty] (np2) at (4+13,-0.5) {2};    
        \node (varphip1) at (2+13,2.7) {$\varphi_1' = 16$};
        \node (varphip2) at (4+13,2.7) {$\varphi_2' = 16$};
    
        \foreach \x in {1, 2}{
            \draw (vp\x) edge[edgeSty]  node[left, pos=0.1] {$\x$} (np1);
            \draw (vp\x) edge[edgeSty]  node[right, pos=0.1] {$\x$} (np2);
        }
            
    \end{tikzpicture}
    }
    \caption{Illustration of Theorem~\ref{thm:get_w_is_hard}. The \textsc{3-partition} instance is $A=\{5,5,5,5,6,6\}$ where $T=16$ and $k=2$. The corresponding instance $(f,G,\varphi)$ of \textsc{Determine-Weight} is shown on the left, where the values of $f$ are defined by the representation $(G',w',\varphi')$ shown on the right. A solution of \textsc{3-partition} is $A_1= \{ 5,5,6 \}$ and $A_2= \{ 5,5,6 \}$, while a solution $w$ of \textsc{Determine-Weight} is shown by the thick orange numbers on the edges on the left.}
    \label{fig:reduction}
\end{figure}

We claim that $A$ is a \texttt{Yes} instance of \textsc{3-partition} if and only if $(f,G,\varphi)$ is a \texttt{Yes} instance of \textsc{Determine-Weight}. To see the forward direction, let  $A=A_1\cup\dots\cup A_k$ be a partition of $A$ into disjoint triplets such that $\sum_{a\in A_i}a=T$ for $i\in[k]$. For each $i\in[k]$ and $a\in A_i$, define $w(a,j) = w'(i,j)$ for $j=1,2$. Then, $(G',w',\varphi')$ is obtained from $(G,w,\varphi)$ by contractions, hence they represent the same assignment valuation, namely $f$.

For the backward direction, let $w\colon V\times N\to\mathbb{R}$ be a weight function such that $(G,w,\varphi)$ represents $f$. Note that $(G', w',\varphi')$ is the unique irreducible representation of $f$ due to the choice of $w'$. This implies that $(G',w',\varphi')$ can be obtained from $(G,w,\varphi)$ by contractions. Since $|V|=3k$, $|V'|=k$, $T/4 < a_i < T/2$ and $\varphi'(i)=T$ for $i\in[k]$, each vertex in $V'$ is obtained by contracting three vertices in $V$ whose $\varphi$ values sum up to $T$, thus leading to a solution of the \textsc{3-partition} problem. This finishes the proof of $\NP$-completeness.

Observe that the maximum value of $f$ can be bounded by $\sum_{i=1}^{k} [\varphi'(i)(w'(i,1)+w'(i,2))]= 2T\sum_{i=1}^k i= (k+1)\sum_{i=1}^{3k} a_i$. That is, the values of $f$ and $\varphi$ are polynomial in the size of $T_\Phi$, hence \textsc{Determine-Weight} remains $\NP$-complete even if the input is represented in unary, concluding the proof of the theorem.
\end{proof}

Let us now define the optimization version of \textsc{Determine-Weight}.

\searchprob{Determine-Opt-Weight}{A function $g\colon T_\Phi \to \mathbb{R}$, a graph $G=(V,N;V\times N)$, and a supply function $\varphi\colon V \to \mathbb{R}_{> 0}$ with $\sum_{i\in V}\varphi(i)=\Phi$.}{Find a weight function $w\colon V\times N\to\mathbb{R}$ such that $\|f-g\|$ is minimized, where $f$ is the bivariate assignment valuation represented by $(G,w,\varphi)$.}

Theorem~\ref{thm:get_w_is_hard} immediately implies the following result.

\begin{cor}\label{thm:invSupply_is_N-Phard}
    \textsc{Determine-Opt-Weight} is strongly $\NP$-hard.
\end{cor}
\begin{proof}
Let $(f,G,\varphi)$ be an instance of \textsc{Determine-Weight}. The same  instance can be considered as an input of \textsc{Determine-Opt-Weight} as well, with $f$ playing the role of $g$. Observe that the optimal value of the latter problem is $0$ if and only if the former problem admits a solution, hence the statement follows by Theorem~\ref{thm:get_w_is_hard}.
\end{proof}

%-----------------------------------------------------------------------
\subsection{The \texorpdfstring{$\MN$}{Mn}-concave case}
\label{sec:mn}
%-----------------------------------------------------------------------

Since every assignment valuation function is $\MN$-concave, the following problem provides a natural relaxation of \textsc{Determine-Opt-Weight}.

\searchprob{Inverse-$\MN$-concave-Function}{A function $g\colon T_\Phi \to \mathbb{R}$ for some $\Phi\in\mathbb{Z}_{>0}$.}{Find an $\MN$-concave function $f\colon T_\Phi\to\mathbb{R}$ such that $\|f-g\|$ is minimized.}

\newpage
\begin{thm}\label{thm:LP_MN_function}
    \textsc{Inverse-$\MN$-concave-Function} can be solved in polynomial time when $\|\cdot\|$ is the $\ell_1$- or $\ell_\infty$-norm.
\end{thm}
\begin{proof}
By Proposition~\ref{prop:MN_concave_function_ineq_characterization}, the problem can be formulated as a linear program with variables $f(k,h)\in\mathbb{R}$ for $(k,h)\in T_\Phi$:
\begin{align*}
    \min\  &\| f-g\| &&\\
    \text{s.t.}~& f(k,h) + f(k+1,h+1) \leq f(k+1,h) + f(k,h+1) &&\forall (k,h) \in T_{\Phi-2},\\
    &f(k,h+1) + f(k+2,h) \leq f(k+1,h+1) + f(k+1,h) &&\forall (k,h) \in T_{\Phi-2},\\
    &f(k+1,h) + f(k,h+2) \leq f(k+1,h+1) + f(k,h+1) &&\forall (k,h) \in T_{\Phi-2}.
\end{align*}    
Note that the number of variables is $|T_\Phi|$ while the number of constraints is $3|T_{\Phi-2}|$. Strictly speaking, this is not an linear program yet as the objective function is not linear. However, a standard and folklore argument shows that the problem can be written as a linear program for both the $\ell_1$- and $\ell_\infty$-norm objectives while increasing the numbers of variables and constraints by at most $|T_\Phi|$ and $2|T_\Phi|$, respectively. This concludes the proof of the theorem.
\end{proof}

%-----------------------------------------------------------------------
\section{Hexagonalizations\label{sec:inverse_hexagons}}
%-----------------------------------------------------------------------

The main message of Proposition~\ref{prop:relation_w-p_and_varphi-edgeLength} is that an assignment valuation $f$ is uniquely determined by its set of two-dimensional maximizer sets of positive-type; see \cite[Lemma 3.4]{otsuka2023assignment}. This gives the idea to consider problems where the input is not the function itself but a decomposition of its domain into $\MN$-convex sets or, as we call them, hexagons.  

For a non-negative integer $\Phi\in\mathbb{Z}_{>0}$, a \emph{hexagonalization} of $T_\Phi$ is a partition $\cH = \{ H_i\}_{i=1}^q$ of $T_\Phi$ into $q$ internally disjoint $\MN$-convex sets $H_i \in \mathbb{Z}^2$. We refer to the members of $\cH$ as \emph{hexagons}; note that this is in line with Proposition~\ref{prop:sets_as_ineq} and its implications as described in Section~\ref{sec:preliminaries}. 
%We denote by $H_0$ the (degenerate) hexagon that contains the point $(0,0)$. 
Since each hexagon is $\MN$-convex, those can be defined as in Proposition~\ref{prop:sets_as_ineq}, and therefore the notion of being positive-, negative- or zero-type can be extended. A hexagonalization $\cH = \{ H_i\}_{i=1}^q $ is said to be \emph{feasible} if every hexagon in $\cH$ is of positive- or zero-type, and there exists at least one hexagon of positive-type; see Figure~\ref{fig:feasible_hexag} for an example.

Note that, by Proposition~\ref{prop:characterization_assig_val}, the family of two-dimensional maximizer sets of any assignment valuation forms a feasible hexagonalization. Recall that $\cH^\cup$ denotes the union closure of $\cH$, while for an $\MN$-concave function $f$, $\mathcal{D}_f$  denotes its family of two-dimensional maximizers. We will need the following technical lemma.

\begin{lem}\label{lem:closure}
Let $\cH$ be a feasible hexagonalization and let $H\in\cH^\cup$ be an $\MN$-convex set. Then, $H$ is of positive- or zero-type. Furthermore, if at least one of the hexagons in $\cH$ whose union is $H$ is of positive-type, then $H$ is of positive-type. 
\end{lem}
\begin{proof} 
Let $H\in\cH^\cup$ be a hexagon and let $H_1,\dots,H_k$ denote the hexagons in $\cH$ whose union is $H$. Note that the right- and left-vertical edges of any of these hexagons are covered by the left- and right-vertical edges of the others, respectively, except for those edges that appear on the boundary of $H$; see Figure~\ref{fig:feasible_hexag} for an example. Hence, since $H_i$ is of positive- or zero-type, a simple counting argument shows that
\begin{align*}
    \ell_{LV}(H)-\ell_{RV}(H)
    {}&{} = \left [\sum_{i=1}^k \ell_{LV}(H_i)\right ]-\left [\sum_{i=1}^k \ell_{RV}(H_i)\right ]\\
    {}&{} =\sum_{i=1}^k [\ell_{LV}(H_i)-\ell_{RV}(H_i)]\\
    {}&{}\geq 0,
\end{align*}
implying that $H$ is also of positive- or zero-type. The second half of the statement follows by observing that if $H_i$ is of positive-type for some $i\in[k]$, then the inequality is strict.
\end{proof}

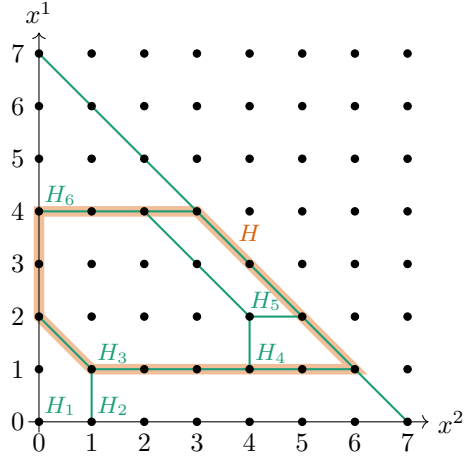
\begin{figure}[t]
    \centering
        \begin{tikzpicture}[scale=.7]
            \draw[MyOrange!40, line width = 4pt] (0,4) -- (3,4) --  (6,1) -- (1,1) -- (0,2) -- cycle; 
            
            \draw[domain=0:7, greenish, thick] plot (\x,{7-\x });         
            \draw[greenish, thick] (0,4) -- (3,4);        
            \draw[greenish, thick] (2,4) -- (4,2);
            \draw[greenish, thick] (4,2) -- (5,2);
            \draw[greenish, thick] (4,2) -- (4,1);
            \draw[greenish, thick] (0,2) -- (1,1);
            \draw[greenish, thick] (1,1) -- (6,1);
            \draw[greenish, thick] (1,1) -- (1,0);

            \draw[->] (-0.2,0) -- (7.4,0) node[right] {$x^2$};
            \draw[->] (0,-0.2) -- (0,7.4) node[above] {$x^1$};
            
            \foreach \x in {0,...,7}{%
                \node at (\x,-0.4) {\x};%x labels
                \node at (-0.4,\x) {\x};%y labels
                \foreach \y in {0,...,7}
                    {\draw [fill = black] (\x,\y) circle (2pt);}
                }    
                
            \node[greenish] at (0.4, 0.3) {\small $H_1$};
            \node[greenish] at (1.4,0.3) {\small$H_2$};
            \node[greenish] at (1.4,1.3) {\small$H_3$};
            \node[greenish] at (4.4,1.3) {\small$H_4$};  
            \node[greenish] at (4.3,2.3) {\small$H_5$};  
            \node[greenish] at (0.4,4.3) {\small$H_6$};
            \node[MyOrange] at (4,3.6) {\small$H$};
        \end{tikzpicture}
        \caption{Example for a feasible hexagonalization $\cH=\{H_1,\dots,H_6\}$. The hexagon $H$ with a yellow border is in $\cH^\cup$. Note that $H_3$ and $H_4$ are of positive-type, and so is $H$.}
        \label{fig:feasible_hexag}
\end{figure}

We are interested in the following problem.

\searchprob{Inverse-Hexagon}{A function $g\colon T_\Phi\to $ for some $\Phi\in\mathbb{Z}_{>0}$ and a feasible hexagonalization $\cH = \{ H_i\}_{i=1}^q $ of $T_\Phi$.}{Find an assignment valuation $f\colon T_\Phi\to\mathbb{R}$ such that $\mathcal{D}_f\subseteq \cH^\cup$ and $\|f-g\|$ is minimized.}

Using the structural observations established in~\cite{otsuka2023assignment}, we formulate this problem as a linear program. The idea is the following: the $\MN$-concavity condition is encoded with the help of Proposition~\ref{prop:MN_concave_function_ineq_characterization}, while the solution being an assignment valuation is enforced using Proposition~\ref{prop:characterization_assig_val} and Lemma~\ref{lem:closure} by ensuring that the two-dimensional maximizer sets form a coarsening of $\cH$. For the latter, recall that for an $\MN$-concave function $f$ and a maximizer set $D=D_f(p)$ corresponding to slope vector $p$, the value of $f(x)-\langle p,x\rangle$ is constant on $D$.

\begin{thm}
\textsc{Inverse-Hexagon} can be solved in polynomial time when $\|\cdot\|$ is the $\ell_1$- or $\ell_\infty$-norm.
\end{thm}
\begin{proof}
Consider the following problem with variables $f(k,h)\in\mathbb{R}$ for $(k,h)\in T_\Phi$, $d_i\in\mathbb{R}$ for $i\in[q]$, and $p_i\in\mathbb{R}^2$ for $i\in[q]$:
     \begin{align}
          \min\  &\| g-f\| && \nonumber\\
        \text{s.t.}~& f(k,h) + f(k+1,h+1) \leq f(k+1,h) + f(k,h+1) && \hspace{-0.2cm}\forall (k,h) \in T_{\Phi-2}, \label{eq:LP_hexagons01}\\
        & f(k,h+1) + f(k+2,h) \leq f(k+1,h+1) + f(k+1,h) && \hspace{-0.2cm}\forall (k,h) \in T_{\Phi-2}, \label{eq:LP_hexagons02}\\
        & f(k+1,h) + f(k,h+2) \leq f(k+1,h+1) + f(k,h+1)  && \hspace{-0.2cm}\forall (k,h) \in T_{\Phi-2}, \label{eq:LP_hexagons03}\\
        & f(x) - \langle p_i, x \rangle= d_i  && \hspace{-0.2cm}\forall x \in  H_i, \forall i\in [q],   \label{eq:LP_hexagons04}\\
        & f(x) - \langle p_i, x \rangle\leq d_i  && \hspace{-0.2cm}\forall x \in  H_j, \forall i,j\in [q],   \label{eq:LP_hexagons05}\\
        & f(0,0)=0.\label{eq:LP_hexagons06}
    \end{align}    
    Here, constraints \eqref{eq:LP_hexagons01}--\eqref{eq:LP_hexagons03} ensure that the solution $f\colon T_\Phi\to\mathbb{R}$ is an $\MN$-concave function. Constraint~\eqref{eq:LP_hexagons04} makes sure that the value of $f(x)-\langle p_i,x_i\rangle$ is a constant on $H_i$ for each $i\in [q]$, while constraint~\eqref{eq:LP_hexagons05} leads to the elements of $H_i$ maximizing $f(x)-\langle p_i,x_i\rangle$. Finally, constraint~\eqref{eq:LP_hexagons06} encodes a property that holds for every assignment valuation. By Proposition~\ref{prop:characterization_assig_val}, these together imply that the resulting $f$ is an assignment valuation as required.
    
    Note that the number of variables is $|T_\Phi| + 3q$ and the number of constraints is bounded from above by $3|T_{\Phi-2}| + 6|T_\Phi|+q|T_\Phi|$. Indeed, it is not difficult to check that every $x\in T_\Phi$ is contained in at most six members of $\cH$. Similarly to the proof of Theorem~\ref{thm:LP_MN_function}, the above is not an linear program yet as the objective function is not linear, but the problem can be written as a linear program for both the $\ell_1$- and $\ell_\infty$-norm objectives while increasing the numbers of variables and constraints by at most $|T_\Phi|$ and $2|T_\Phi|$, respectively. This concludes the proof of the theorem.
\end{proof}

%-----------------------------------------------------------------------
\section{Conclusions}

In this paper, we considered a natural inverse problem related to bivariate assignment valuations. Building on the work of Otsuka and Shioura~\cite{otsuka2023assignment}, we derived hardness for the problem, while we provided linear programming formulations in certain settings that lead to efficient algorithms for solving those. Analogous problems can be formulated with the additional requirement that the function $f$ has to take integer values. For such problems, the proposed linear programming formulations do not suffice, hence understanding their complexities is an interesting research direction.

\medskip

\paragraph*{Acknowledgement.} The work was supported by the Lend\"ulet Programme of the Hungarian Academy of Sciences -- grant number LP2021-1/2021, by the Ministry of Innovation and Technology of Hungary -- grant number ELTE TKP 2021-NKTA-62, and by Dynasnet European Research Council Synergy project -- grant number ERC-2018-SYG 810115.

%%%%%%%%%%%%%%%%%%%%%%%%%%%%%%%%
\bibliographystyle{abbrv}
\bibliography{inverse_assignment.bib}

\end{document}